\documentclass[11pt,twoside]{amsart}
\usepackage{amssymb}

\newtheorem{thm}{Theorem}[section]

\newtheorem{lem}[thm]{Lemma}
\newtheorem{cor}[thm]{Corollary}
\newtheorem{defn}[thm]{Definition}

\newtheorem{ex}[thm]{Example}
\newtheorem{rem}[thm]{Remark}

\newcommand{\skipit}[1]{{}}
\newcommand{\prfend}{\hbox to7pt{\hfil}
\par\vskip-\baselineskip\hbox to\hsize
{\hfil\vbox {\hrule width6pt height6pt}}\vskip\baselineskip}

\newcommand{\myarrow}[2]{\hbox to #1pt{\hfil$\to$\hfil}{\hskip-#1pt{\raise
10pt\hbox to#1pt{\hfil$\scriptscriptstyle #2$\hfil}}}}

\begin{document}

\title{\Large Forcing the Strong Lefschetz \\and the Maximal Rank Properties}
\author{\large Fabrizio Zanello \tiny and \large Jeffery V. Zylinski}
\address{Department of Mathematical Sciences, Michigan
Technological University, Houghton, MI 49931}
\email{zanello@math.kth.se, jvzylins@mtu.edu}

\maketitle
\markboth{F. Zanello and J. Zylinski}{Forcing the Strong Lefschetz and the Maximal Rank Properties}
\begin{abstract}
Three basic properties which standard graded artinian $k$-algebras may or may not enjoy are the Weak and Strong Lefschetz Properties and the Maximal Rank Property (respectively WLP, SLP, and MRP). 

In this paper we will assume that the base field $k$ has characteristic zero. It is known that SLP implies MRP, which in turn implies WLP, but that both implications are strict. However, it surprisingly turned out (see \cite{HMNW}) that the set of Hilbert functions admitting {\em any} algebras with WLP coincides with the corresponding set for SLP (and therefore with that for  MRP). In \cite{MZ1}, J. Migliore and the first author, using Green's theorem and a result of Wiebe, characterized the Hilbert functions forcing {\em  all} algebras to enjoy WLP. The purpose of this note is to prove the corresponding characterizations for both SLP and MRP. Unsurprisingly (or surprisingly??), the two characterizations coincide, but they define a class of Hilbert functions strictly smaller than that determined for WLP in \cite{MZ1}.

Our methods include the Herzog-Popescu theorem on quotients of $k$-algebras modulo a general form, a result of Wiebe, and gins  and stable ideals. At the end, we will also discuss the importance of assuming that the characteristic be zero, and we will exhibit a class of codimension 2 monomial complete intersections for which SLP (but not MRP) fails in positive characteristic.
\end{abstract}
\large
{\ }\\
\section{Introduction}

Let $A= \bigoplus_{i \geq 0} A_i=R/I$ be a standard graded $k$-algebra, where $R=k[x_1,...,x_r]$ is a polynomial ring in $r$ variables over a field $A_0=k$ of characteristic zero, $I$ is a homogeneous ideal of $R$, and $\deg(x_i)=1$ for all $i$ (i.e., $R=k[R_1]$, and therefore $A=k[A_1]$). The {\em Hilbert function} of $A$ is the function $H=H(A):$ $\mathbb{N}{\ }\rightarrow {\ } \mathbb{N}$ defined by $H(A,d) = \dim_k A_d$. $A$ is {\em artinian} if the ideal $I$ is $(x_1,...,x_r)$-primary, or equivalently, if the Hilbert function of $A$ is eventually 0.

The following are three very natural properties that an artinian algebra $A$ may enjoy: $A$ is said to have the {\em Weak Lefschetz Property} (WLP) if the map $\times L: A_i \rightarrow A_{i+1}$, where $L$ is a {\em general} linear form (that is, its coefficients belong to a  Zariski-open subset of the base field $k$), has maximal rank for all $i\geq 0$;  $A$ has the {\em Maximal Rank Property} (MRP) if, for any $d\geq 1$, the map $\times F: A_i \rightarrow A_{i+d}$,  where $F$ is a general form of degree $d$, has maximal rank for all $i\geq 0$; finally, $A$ has the {\em Strong Lefschetz Property} (SLP) if, for any $d\geq 1$, the map $\times L^d: A_i \rightarrow A_{i+d}$,  where $L$ is a general linear form, has maximal rank for all $i\geq 0$.

While both the Weak and Strong Lefschetz Properties have been extensively investigated in the literature (for a broad overview and the main results obtained so far, see for instance \cite{BZ,GHMS,HMNW,MM} and their bibliographies), the Maximal Rank Property - perhaps the most natural intermediate property between WLP and SLP - has only been more recently introduced in \cite{MM} by J. Migliore and R. Mir\'o-Roig. Indeed, SLP implies MRP by semicontinuity (since MRP is determined, by definition, by the existence of certain general forms, and these are open conditions - see \cite{MM}), and MRP is clearly stronger than WLP. It is known that none of the opposite implications hold true: In particular, \cite{MM} contains a very  interesting example of a level, type 3 algebra in characteristic zero enjoying MRP but not SLP: $k[x,y,z]/(x^3,y^3,z^3,(x+y+z)^3).$ (We only remark here that this example can be naturally extended to codimension 3 level algebras of any type $t\geq 3$.) 

However, the class of algebras enjoying MRP {\em but not} SLP does seem to be ``quite small'' in characteristic zero (unlike in positive characteristic - see Example \ref{eee}), even if the relationship between this two properties is still mostly unclear. Certainly, for instance, as we will see follows later from a theorem of Herzog-Popescu,  the two properties are equivalent for algebras quotients of  lex-segment ideals and, more generally, of stable ideals. One of the goals of this paper is to shed some light on the relationship intercurrent between MRP and SLP.

A broad and interesting problem in this area is to determine structural results on the Hilbert functions of artinian algebras enjoying any of the three properties above: See, for instance, \cite{HMNW,MM,MMN,MZ1,zanello2}, or for the same problem with a more specific focus on special classes of  algebras, such as Gorenstein or level, \cite{boij2,BZ,GHMS,harima,ikeda,Mi,MN,MZ2,Wi,zanello2}.

In particular, in \cite{HMNW}, the set of Hilbert functions admitting {\em at least one} algebra with WLP was characterized, as well as the corresponding set for SLP. Curiously, it turned out that these two sets coincide (and therefore coincide with that for MRP). In \cite{MZ1}, Migliore and the first author, instead, taking a different approach, characterized the Hilbert functions forcing {\em  all} algebras to enjoy WLP. In this note, we will prove the corresponding characterization for both SLP and MRP. Perhaps even more curiously, these two characterizations coincide, but they define a class of Hilbert functions strictly contained in that determined for WLP.

We will also see that the class of Hilbert functions forcing SLP or MRP on all algebras coincides with the class of Hilbert functions simply forcing the corresponding property on lex-segment algebras, just like with WLP (\cite{MZ1}). However, whereas it is also true that any Gotzmann algebra enjoying WLP forces all algebras with the same Hilbert function to enjoy it, this stronger fact is false with respect to SLP or MRP.

In the next section we will state some preliminary results  needed in Section 3, which contains the two characterizations (Theorems \ref{mrp} and \ref{slp}). We will be employing a few different methods, including the Herzog-Popescu theorem on the Hilbert functions of quotients of $k$-algebras modulo a general form \cite{HP}, a result of Wiebe \cite{Wi}, gins and stable ideals. We will need characteristic zero for some key arguments of this paper; at the end we will also discuss the importance of this assumption and show that even the simplest examples of codimension 2 monomial complete intersections - which all enjoy SLP in characteristic zero - may lose that property (but not MRP) in passing to positive characteristic.

\section{Preliminary results}

Let us now introduce some of the main definitions and results we will need in this note.

\begin{defn} \label{ddd} Let $n$ and $d$ be positive integers. The {\em $d$-binomial expansion of n} is $$n=n_{(d)}=\binom{n_d}{d}+\binom{n_{d-1}}{d-1}+\cdots +\binom{n_j}{ j},$$ where $n_d>n_{d-1}>...>n_j\geq j\geq 1$. Note that, under these hypotheses, the $d$-binomial expansion of $n$ is unique (e.g., see \cite{BH}, Lemma 4.2.6). Also, define, for $a$ and $b$ integers,
$$(n_{(d)})^a_b=\binom{n_d+a}{d+b}+\binom{n_{d-1}+a}{ d-1+b}+\cdots +\binom{n_j+a}{j+b},$$
where, as usual, $\binom{m}{n}=0$ whenever $m<n$ or $n<0$. Finally, for $0\leq c<d$, let $$n_{((d,c))}=\binom{n_d-c-1}{ d-c}+\binom{n_{d-1}-c-1}{ d-1-c}+\cdots +\binom{n_q-c-1}{ q-c}+\binom{n_c-c}{ 0},$$
where we define $q=j$ if $j>c$ and $q=c+1$ if $j\leq c$, and $\binom{n_c-c}{ 0}$ is set to equal 0 for $c<j$.
\end{defn}

The possible Hilbert functions that may occur for standard graded algebras are characterized by a well-known theorem of Macaulay's (see, e.g., \cite{BH}, Theorem  4.2.10); those possible sequences of integers are called {\em $O$-sequences}. It is easy to see  that an $O$-sequence having, in some degree $d$, a value $\leq d$, is  non-increasing from that degree on. Also, the Hilbert functions of artinian algebras are characterized as those $O$-sequences which are eventually 0.

From \cite{MZ1} we have the following characterization of the Hilbert functions forcing  WLP:

\begin{thm}(\cite{MZ1}, Theorem 5)\label{mz}
Let $H$: $1,h_1,h_2,...,h_e, h_{e+1}=0$ be an $O$-sequence, and let $t$ be the smallest integer such that $h_t\leq t$. Then {\em all} the artinian algebras having Hilbert function $H$ enjoy WLP if and only if, for all indices $i=1,2,...,t-1$, we have $$h_{i-1}=((h_i)_{(i)})^{-1}_{-1}.$$
\end{thm}

\begin{rem}\label{rr}
i) Let $I$ be a homogeneous ideal of $R$  and  let $F\in R$ be any form of degree $d$ (not in $I$). Notice that there is the following exact sequence:
\begin{equation} \label {exact}
0 \rightarrow R/(I:F)(-d) \stackrel{\times F}{\longrightarrow} R/I \rightarrow R/(I,F) \rightarrow 0.
\end{equation}
ii) From (\ref{exact}) it follows that the Hilbert function of an algebra $A=R/I$ decomposes as the sum of the Hilbert function of $R/(I:F)$ (shifted by $d$ degrees to the right) and that of $R/(I,F)$.\\
iii) From the previous point we have that, if $A$ is artinian of {\em socle degree} $e$ (that is, $e$ is the last degree where $A_e\neq 0$), then the socle degree of $R/(I:F)$ is  $\leq e-d$. Notice, moreover, that if $F$ is the $d$-th power of a general linear form (or, by semicontinuity,  it is  a general form of degree $d$), then the socle degree of $R/(I:F)$ is {\em exactly} $e-d$. This easily follows from the classical fact that, {\em in characteristic zero}, $R_e$ is spanned by $e$-th powers of linear forms.
\end{rem}

We now rephrase  SLP and  MRP, in view of the exact sequence (\ref{exact}):

\begin{lem}\label{reph}
i) Let $A=R/I$ be an artinian algebra  with Hilbert function $H$: $h_0=1,h_1,h_2,...,\\h_e,h_{e+1}=0$. Then $A$ enjoys  SLP if and only if, for any index $d$ and for the $d$-th power $L^d$ of a general linear form $L$, we have, for all $i\geq 0$:
$$H(A/L^d,i)=\max \{h_i-h_{i-d} ,0\},$$
where we set $h_j=0$ for $j<0$.\\
ii) Similarly, $A=R/I$ enjoys  MRP if and only if, for any index $d$ and for a general form  $F$ of degree $d$, we have, for all  $i\geq 0$:
$$H(A/F,i)=\max \{h_i-h_{i-d} ,0\}.$$
\end{lem}

\section{The main results}

We are now ready for the first main result of this note, namely the characterization of the Hilbert functions forcing  MRP:

\begin{thm}\label{mrp}
Let $H$: $1,h_1=r,h_2,...,h_e, h_{e+1}=0$ be an $O$-sequence, and let $t$ be the smallest integer such that $h_t\leq t$. Then all the artinian algebras having Hilbert function $H$ enjoy  MRP if and only if:\\
i) $r=2$; or\\
ii) $r>2$, $h_t\leq 2$, and, for all indices $i=1,2,...,t-1$, $$h_{i-1}=((h_i)_{(i)})^{-1}_{-1}.$$
\end{thm}

\begin{proof} Let us first show that if, given an artinian Hilbert function $H$,  MRP holds for all algebras $A$ with that Hilbert function, then $H$ has the form i) or ii). Suppose $r>2$. Since  MRP implies  WLP, by Theorem \ref{mz}
we have that $h_{i-1}=((h_i)_{(i)})^{-1}_{-1},$ for $i=1,2,...,t-1$. Hence it remains to show that $h_t\leq 2$. Suppose $ h_t\geq 3$. Thus, it suffices to prove that some algebra with Hilbert function $H$ fails to have MRP, and we will use $R/I$, where $I$ is the lex-segment ideal for the Hilbert function $H$.

This can be done in at least two ways: The first uses Wiebe's characterization for the Hilbert functions whose lex-segment algebras have SLP (Theorem 4.4 of \cite{Wi}). Indeed, that numerical characterization is the same as that of the statement of this theorem, and moreover, a result of Herzog-Popescu (\cite{HP}, Lemma 1.4) says that $x_r^d$ is a general form of degree $d$ with respect to stable ideals, and so in particular lex-segment ideals. Hence, by Lemma \ref{reph}, when it comes to lex-segments,  SLP is equivalent to MRP, and we are done. A second proof of this fact  is more direct and also illustrates the idea behind Wiebe's approach: Since $3\leq h_t\leq t$ and $I$ is lex-segment, then in degree $t$, $R/I$ is equal to $(R/I)_t=\langle \overline{x_{r-1}^{h_t-1}x_r^{t-h_t+1}},...,\overline{x_{r-1}x_r^{t-1}},\overline{x_r^t}  \rangle $. In particular, only the last two variables appear. Thus, it is easy to see that the multiplication by any form $F\in R_{t-1}$ between $A_1$ and $A_t$ has (the images of) all the variables in its kernel, except possibly for the the last two. Therefore, since $\dim_kA_t\geq 3$, surjectivity fails for that map, and so MRP for $R/I$, as desired.

Let us now suppose that the Hilbert function $H$ has the form described in i) or in ii). We want to prove that all algebras $A$ with such Hilbert function enjoy  MRP. If $r=2$, this is well-known, since all algebras (in characteristic zero) even have SLP (see \cite{HMNW}, Proposition 4.4). Thus, suppose  that $r>2$, and let $F\in R$ be a general form of any degree $d\geq 1$. We want to show that the map $\times F$ between $A_i$ and $A_{i+d}$ has maximal rank, for each $i\geq 0$.

Consider first the case $i+d\leq t-1$. Notice that the Hilbert function of $A$ is increasing until degree $t-1$, since by \cite{BG}, it follows that $((h_i)_{(i)})^{-1}_{-1}<h_i$; hence we want to show that  $\times F$ is injective in these degrees. Let $F=L^d$, where $L$ is a general linear form. Since by Theorem \ref{mz}, WLP holds for $A$, we have that each of the maps $\times L$ between $A_i$ and $A_{i+1}$, $A_{i+1}$ and $A_{i+2}$, ..., $A_{i+(d-1)}$ and $A_{i+d}$ are injective. Thus, $\times L^d$ is also injective, and so $A$ enjoys  SLP in these degrees. It follows by semicontinuity that $A$ also enjoys  MRP, as desired.

Now suppose that $t\leq i+d\leq e$. Hence, since $h_t\leq 2$, by Macaulay's theorem we have $h_{i+d}=\dim_kA_{i+d}\leq 2$. If $i=0$, then the map $\times F$ is clearly injective for any form $F$ of degree $d$ whose image is non-zero in $A_d$, and therefore  MRP between degrees 0 and $d$ again follows by semicontinuity.

Thus, let $i>0$. Since $h_{i+d}\leq 2$, Macaulay's theorem guarantees that $h_i\geq h_{i+d}$; hence we want to show that the map $\times F$  between $A_i$ and $A_{i+d}$ is surjective, where $F$ is a general form of degree $d$. But, by Lemma \ref{reph}, this is equivalent to proving that $H(A/F,i+d)=0$. A tedious but straightforward computation involving the binomial coefficients of Definition \ref{ddd} shows that under the current assumptions - namely, $h_{i+d}\in \lbrace 1,2\rbrace $ and $i+d> \deg(F)=d$ - we have:
 $$\sum_{0\leq c\leq d-1}H(A,i+d)_{((i+d,c))}=0.$$
(Notice that, expectedly, the same numerical result does not hold for $h_{i+d}\geq 3$.) Therefore, the conclusion immediately follows from Herzog-Popescu's Theorem 3.7 of \cite{HP}, which says that if $A=R/I$ (is any standard graded $k$-algebra) and $F\in R$ is a general form of degree $d$, then the Hilbert function of $A/F\simeq R/(I,F)$ satisfies, for all integers $p\geq d$:
$$H(A/F,p)\leq \sum_{0\leq c\leq d-1}H(A,p)_{((p,c))}.$$
\end{proof}

We are now ready  to prove the second main result of this note, namely the characterization of the Hilbert functions forcing SLP. As we noticed above, it will turn out to be exactly the same as that for  MRP, which instead describes a strictly smaller set of Hilbert functions than those forcing WLP. Curiously, even if SLP is a stronger condition than MRP, we will need the existence of MRP proved in the previous theorem to deduce SLP in some degrees.

\begin{thm}\label{slp}
Let $H$: $1,h_1=r,h_2,...,h_e, h_{e+1}=0$ be an $O$-sequence, and let $t$ be the smallest integer such that $h_t\leq t$. Then all the artinian algebras having Hilbert function $H$ enjoy SLP if and only if:\\
i) $r=2$; or\\
ii) $r>2$, $h_t\leq 2$, and, for all indices $i=1,2,...,t-1$, $$h_{i-1}=((h_i)_{(i)})^{-1}_{-1}.$$
\end{thm}

\begin{proof} The ``only if'' part immediately follows from Theorem \ref{mrp} and the fact that  SLP is stronger  than MRP. Hence suppose that $H$ has the form described in i) or ii), and let $A=R/I$ be any algebra with Hilbert function $H$. We want to show that $A$ has  SLP. 

Consider the generic initial ideal of $I$, say $J$, with respect to the reverse lexicographic order. It is  well-known that $R/I$ and $R/J$ have the same Hilbert function, and that $J$ is Borel-fixed. Since we are in characteristic zero, $J$ is also strongly stable and hence stable (see, e.g., \cite{Ei}, Section 15.9). Moreover, Wiebe showed in \cite{Wi}, Proposition 2.8 that $R/I$ enjoys  SLP if and only if $R/J$ does.

Thus, it suffices to show that all stable ideals with Hilbert function $H$ enjoy SLP. But we know that they all enjoy  MRP by Theorem \ref{mrp}; moreover, by \cite{HP}, Lemma 1.4, $x_r^d$ is a general form of degree $d$ with respect to stable ideals. Therefore, from Lemma \ref{reph} it immediately follows that SLP holds with respect to the powers of $L=x_r$, and the proof is complete.
\end{proof}

From \cite{Wi}, Theorem 4.4 and what we have observed earlier, we easily deduce the following interesting corollary  to our theorems:

\begin{cor}\label{equiv}
All artinian algebras with a given Hilbert function $H$ have SLP if and only if they all have MRP, if and only if the lex-segment algebra corresponding to $H$ has SLP, if and only if the lex-segment algebra corresponding to $H$ has MRP.
\end{cor}

\begin{rem}
The same equivalent condition with respect to the lex-segment also holds for WLP, as proved in \cite{MZ1}, Corollary 6. There, more was actually shown to be true, namely that all algebras having Hilbert function $H$ enjoy  WLP if and only if {\em any Gotzmann} algebra having Hilbert function $H$ does. Interestingly, that stronger result, instead, does not extend here: Indeed, as proved in \cite{Wi}, Example 3.4, one has that $I'=(x^2,xy,y^3,y^2z,xz^3,yz^3,z^4)\subset R=k[x,y,z]$ is a Gotzmann ideal, and that $A=R/I'$ enjoys SLP (and therefore MRP). However, the Hilbert function of $A$ is $H: 1,3,4,3,0$, which does not satisfy the hypotheses of Theorems \ref{slp} and \ref{mrp}. Therefore, the algebra $R/I$, where $I$ is  the lex-segment ideal corresponding to $H$, does not enjoy SLP or MRP.
\end{rem}

We conclude by briefly discussing our assumption that the characteristic of the base field be zero. Not only has this been necessary in a few key arguments, but as the next simple example shows, some results are  far from being true characteristic-free.

\begin{ex}\label{eee}
In positive characteristic, SLP may fail in codimension 2, even in the simplest instances of monomial complete intersections: In fact, all algebras $A^{(p)}=k[x,y]/(x^p,y^p)$, where char($k$)$=p$, do not enjoy SLP.

This is because the $p$-th power of any linear form $L\in k[x,y]$ belongs to the vector space $\langle x^p,y^p \rangle $, and therefore $L^a$ is contained in the ideal $(x^p,y^p)$ for all $a\geq p$, i.e., the image of those $L^a$ is 0 in $A^{(p)}$. (Notice that this implies the fact that powers of linear forms do not span all of $R_a$ in characteristic $p$, unlike in characteristic zero, as we saw in Remark \ref{rr}.) Thus, for $A^{(p)}$, SLP even fails between degree 0 and any degree $a$ such that $p\leq  a\leq 2(p-1)$.

It can be also shown, interestingly, that  these algebras $A^{(p)}$, although they do not enjoy SLP,  all have MRP.
\end{ex}

More generally, from the latest research done on  WLP and SLP, there seems to consistently emerge a remarkable - and perhaps not entirely expected - difference between the results which hold in characteristic zero and those true in positive characteristic (in fact, the latter appear to be many fewer). So far, research has more systematically focused on the characteristic zero case, and indeed, techniques have been introduced and employed to study the Lefschetz Properties when char($k$)$=0$ also in previous papers (they came from Macaulay's inverse systems, as in \cite{HMNW}; or from geometry, as in \cite{HMNW,MZ1,MZ2} - the latter paper applying to WLP a technical result from \cite{MNZ}).

However, understanding  WLP, SLP  and MRP in characteristic $p$ is also a basic problem, and it will definitely deserve more specific attention in the future. It seems likely that SLP is, of the three Properties, the one which ``loses more'' in passing to positive characteristic; in particular, interestingly, we believe that - as the previous example also suggests - the class of artinian algebras enjoying MRP {\em but not} SLP is ``much larger'' in characteristic $p$ than it is in characteristic zero.\\
{\ }\\
{\bf Acknowledgements.} We thank Juan Migliore and the  referee for helpful comments.


\begin{thebibliography}{ll}

\bibitem{BG} A.M. Bigatti and A.V. Geramita: {\em Level Algebras, Lex Segments and Minimal Hilbert Functions}, Comm. Algebra {\bf 31} (2003), 1427-1451.

\bibitem{boij2}  M. Boij: {\em Components of the space parameterizing graded Gorenstein Artin algebras	with a given Hilbert function}, Pacific J. Math. {\bf 187} (1999), no. 1, 1-11.

\bibitem{BZ}  M. Boij and F. Zanello, {\em Level algebras with bad properties}, Proc. Amer. Math. Soc. {\bf 135} (2007), no. 9, 2713-2722.

\bibitem{BH} W. Bruns and J. Herzog: ``Cohen-Macaulay rings'', Cambridge studies in advanced mathematics {\bf 39} (1998), Revised edition, Cambridge, U.K.

\bibitem{Ei} D. Eisenbud: ``Commutative algebra with a view toward algebraic geometry'', Graduate Texts in Mathematics {\bf 150} (1995), Springer-Verlag, New York.

\bibitem{GHMS}  A.V. Geramita, T. Harima,  J. Migliore and Y. Shin: ``The Hilbert function of a level algebra'', Mem. Amer. Math. Soc. {\bf 186} (2007), no. 872.

\bibitem{harima} T. Harima, {\em Characterization of Hilbert functions of Gorenstein Artin algebras with the Weak Stanley property}, Proc. Amer. Math. Soc. {\bf 123} (1995), 3631-3638.

\bibitem{HMNW} T. Harima, J. Migliore, U. Nagel and J. Watanabe: {\em The Weak and Strong Lefschetz Properties for Artinian $K$-Algebras}, J. Algebra {\bf 262} (2003), 99-126.

\bibitem{HP} J. Herzog and D. Popescu: {\em Hilbert functions and generic forms}, Compositio Math. {\bf 113} (1998), no. 1, 1-22.

\bibitem{ikeda} H. Ikeda, {\em Results on Dilworth and Rees numbers of artinian local rings}, Japan. J. Math. (N.S.) {\bf 22} (1996), 147-158.

\bibitem{Mi} J. Migliore: {\em The geometry of the Weak Lefschetz Property}, Canad. J. Math. {\bf 60} (2008), 391-411.

\bibitem{MM} J. Migliore and R. Mir\'o-Roig: {\em Ideals of general forms and the ubiquity of the Weak Lefschetz property}, J.  Pure Appl. Algebra {\bf 182} (2003), 79-107.

\bibitem{MMN} J. Migliore,  R. Mir\'o-Roig and U. Nagel: {\em Almost complete intersections and the Weak Lefschetz Property}, in preparation.

\bibitem{MN} J. Migliore and U. Nagel: {\em Reduced arithmetically Gorenstein schemes and Simplicial Polytopes with maximal Betti numbers}, Adv. Math. {\bf 180} (2003), 1-63. 

\bibitem{MNZ} J. Migliore, U. Nagel and F. Zanello: {\em A characterization of Gorenstein Hilbert functions in codimension four with small initial degree}, Math. Res. Lett. {\bf 15} (2008), no. 2, 331-349.

\bibitem{MZ1} J. Migliore and F. Zanello: {\em The Hilbert functions which force the Weak Lefschetz Property}, J. Pure Appl. Algebra {\bf 210} (2007) no. 2, 465-471.

\bibitem{MZ2} J. Migliore and F. Zanello: {\em The strength of the Weak Lefschetz Property}, Illinois J. Math., to appear.

\bibitem{Wi} A. Wiebe: {\em The Lefschetz Property for componentwise linear ideals and Gotzmann ideals}, Comm. Algebra {\bf 32} (2004), no. 12, 4601-4611.

\bibitem{zanello2}  F. Zanello, {\em A non-unimodal codimension 3 level $h$-vector}, J. Algebra {\bf 305} (2006), no. 2, 949-956.

\end{thebibliography}
\end{document}